\documentclass[draft]{amsart}

\usepackage{pgf,tikz}
\usetikzlibrary{arrows}
\usepackage{cite}

\usepackage{amssymb,graphicx,enumerate,color}
\usepackage{multirow}

\newcommand{\s}{{\sigma}}

\renewcommand{\l}{\lambda}

\newcommand{\U}{\mathcal{U}}

\newcommand{\beqs}{\begin{equation*}}
\newcommand{\eeqs}{\end{equation*}}
\numberwithin{equation}{section}
 \theoremstyle{plain}
\newtheorem{theorem}{Theorem}[section]

\newtheorem{corollary}[theorem]{Corollary}

\newtheorem{conjecture}[theorem]{Conjecture}

\theoremstyle{remark}

\theoremstyle{empty}

\begin{document}

\makeatletter
\def\imod#1{\allowbreak\mkern10mu({\operator@font mod}\,\,#1)}
\makeatother

\author{Alexander Berkovich}
   \address{Department of Mathematics, University of Florida, 358 Little Hall, Gainesville FL 32611, USA}
   \email{alexb@ufl.edu}

\author{Ali Kemal Uncu}
   \address{Department of Mathematics, University of Florida, 358 Little Hall, Gainesville FL 32611, USA}
   \email{akuncu@ufl.edu}

\title[Some Elementary Partition Inequalities and Their Implications]{Some Elementary Partition Inequalities and Their Implications}



\begin{abstract} 

We prove various inequalities between the number of partitions with the bound on the largest part and some restrictions on occurrences of parts. We explore many interesting consequences of these partition inequalities. In particular, we show that for $L\geq 1$, the number of partitions with $l-s \leq L$ and $s=1$ is greater than the number of partitions with $l-s\leq L$ and $s>1$. Here $l$ and $s$ are the largest part and the smallest part of the partition, respectively.

\end{abstract}   
   


\date{\today}
   
\maketitle

\section{Introduction}

Let $\pi = (1^{f_1},2^{f_2},\dots)$ be a sequence, where all exponents $f_i \in \mathbb{Z}_{\geq 0}$ and all but finitely many of them are zero. We call $\pi$ a \textit{partition} (shown in \textit{frequency representation} \cite{Theory_of_Partitions}), where the exponents $f_i$ are the number of occurrences of $i$. The numbers $i$ with non-zero frequencies in $\pi$ are called \textit{parts} of $\pi$. Since there are only finitely many non-zero frequencies in a partition $\pi$, the sum \[ |\pi| := \sum_{i\geq 1} i\cdot f_i \] is finite. This sum, $|\pi|$, is called the \textit{norm} of the partition $\pi$. To shorten the notation one can ignore the zero frequencies; we keep the option of writing any zero frequencies that need emphasizing. As an example $\pi = (1^4, 3^2,4^0,10^1)$ is a partition of $20$ (meaning $|\pi| = 20$), where $1$ appears as a part with frequency $4$, $3$ appears twice, $4$ is not a part, and part $10$ only appears once in $\pi$. The partition where all the frequencies are equal to zero is a conventional and unique partition of 0.

Let $\delta_{i,j}$ be the standard Kronecker delta function yielding 1 for $i=j$, and vanishing otherwise. We define sets $\mathcal{A}_{L,1}$ and $\mathcal{A}_{L,2}$ for integers $L\geq 1$.
\begin{enumerate}[i.]
\item Let $\mathcal{A}_{L,1}$ be the set of partitions with the smallest part being 1, where all the parts $\leq L+1$ and $f_L = \delta_{L,1}$,\\[-1.5ex]
\item and let $\mathcal{A}_{L,2}$ be the set of non-empty partitions where the parts are in the domain $\{2,3,\dots, L+1 \}$.
\end{enumerate}

These sets satisfy the following relation.

\begin{theorem}\label{Ali_comb_THM} For any $L\geq 2$ and $N\geq 1 $,\vspace{1mm}
\begin{equation}\label{Ali_comb_inequality}|\{\pi:\pi\in\mathcal{A}_{L,1},\ |\pi|=N\}|\geq|\{\pi:\pi\in\mathcal{A}_{L,2},\ |\pi|=N\}|.\end{equation}
\end{theorem}\vspace{1mm}

Elementary combinatorial inequalities, such as \eqref{Ali_comb_inequality}, have interesting implications for $q$-series and the theory of partitions. This simple observation about the magnitude of sets, in this case, implies non-negativity results for a refinement of an earlier discussed weighted partition identity result \cite{BerkovichUncu4}. We introduce that result and its refinement here.

Let $\U$ be the set of partitions with positive norm. We define some natural partition statistics. Let
\begin{enumerate}[i.]
\item $s(\pi)$ denote the smallest part of the partition $\pi$,
\item $l(\pi)$ denote the largest part of $\pi$,
\item $\nu(\pi):= \sum_{i\geq 1} f_i$ denote the total number of parts in $\pi$,
\item $r(\pi):=l(\pi)-\nu(\pi)$, rank of $\pi$.
\end{enumerate}

In \cite{BerkovichUncu4}, we introduced a new partition statistics $t(\pi)$ to be the number defined by the properties
\begin{enumerate}[i.]
\item $f_i\equiv1\mod{2}$, for $1\leq i\leq t(\pi)$, 
\item $f_{t(\pi)+1}\equiv 0 \mod{2}$.
\end{enumerate}  Note that for any $\pi\in \U$ with an even frequency of 1 (where $f_1$ might be 0) we have $t(\pi)=0$. We will refer to $t(\pi)$ as the \textit{length of the initial odd-frequency chain}. With this new statistic the authors have proven a new combinatorial identity of partitions.

\begin{theorem}\label{Ordinary_Partitions_Combinatorial_Weighted_Theorem}
\begin{equation}\label{Ordinary_Partitions_Combinatorial_Weighted_Equation}
\sum_{\pi\in\U} (-1)^{s(\pi)+1} q^{|\pi|} = \sum_{\pi\in\U} t(\pi) q^{|\pi|}.
\end{equation}
\end{theorem}

This generating function identity can be articulated easily as a combinatorial correspondence as follows:

\begin{quote}\textit{The total count of partitions of a positive integer $N$, counted with the weight $1$ if the smallest part is odd, and $-1$ if the smallest part is even, is the same as the total of all odd-frequency chain lengths of partitions of $N$.}\end{quote}

One example of Theorem~\ref{Ordinary_Partitions_Combinatorial_Weighted_Theorem} is given in Table~\ref{Table_1}.

\begin{table}[htb]\caption{Example of Theorem~\ref{Ordinary_Partitions_Combinatorial_Weighted_Theorem} with $|\pi|=6$.}\label{Table_1}
\begin{center}\vspace{-.5cm}
\[\begin{array}{ccr|c}
&\pi\in\U  		&(-1)^{s(\pi)+1} & t(\pi)  \\
&(1^6) 			& 1\hspace*{.5cm}&	0	\\
&(1^4,2^1)		& 1\hspace*{.5cm}&	0		\\
&(1^3,3^1)		& 1\hspace*{.5cm}&	1		\\
&(1^2,2^2) 		& 1\hspace*{.5cm}&	0		\\
&(1^2,4^1) 		& 1\hspace*{.5cm}&	0		\\
&(1^1,2^1,3^1)	& 1\hspace*{.5cm}&	3		\\
&(1^1,5^1) 		& 1\hspace*{.5cm}&	1		\\
&(2^3) 			& -1\hspace*{.5cm}&	0		\\
&(2^1,4^1)		&-1\hspace*{.5cm}&	0	\\
&(3^2)			&1\hspace*{.5cm}&		0	\\
&(6^1)			&-1\hspace*{.5cm}&	0	\\ \hline
Total&			&5 \hspace*{.5cm}&	5	\\
\end{array}\]
\end{center}
\end{table}

We define \textit{non-negativity} of a series \[S=\sum_{n\geq 0} a_n q^n, \] if for all $n$, $a_n\geq 0$, where $q$ is a formal summation variable. We denote the non-negativity by the notation \[S \succcurlyeq 0.\] One important observation about the Theorem~\ref{Ordinary_Partitions_Combinatorial_Weighted_Theorem} is that the statistics $t(\pi)$ is non-negative for any partition $\pi$. It is clear that \[\sum_{\pi\in\U} t(\pi) q^{|\pi|} = \sum_{n\geq 1} p_t(n) q^n \succcurlyeq 0,\] where $p_t(n)$ is the total weighted count of partitions with the $t$ statistics. This implies that the series in \eqref{Ordinary_Partitions_Combinatorial_Weighted_Equation} are non-negative. Written in analytic form, the identity \eqref{Ordinary_Partitions_Combinatorial_Weighted_Equation} is equivalent to \begin{equation}\label{Ordinary_Partitions_Analytic_Identity}
\sum_{n\geq 1} \frac{q^n}{1+q^n}\frac{ 1}{(q;q)_{n-1}} = \sum_{n\geq 1} \frac{q^{n(n+1)/2}}{(q^2;q^2)_n (q^{n+1};q)_\infty},
\end{equation} where \[(a;q)_L := \prod_{i=0}^{L-1}(1-aq^i)\text{ and }(a;q)_\infty := \lim_{n\rightarrow\infty}(a;q)_L\text{ for }|q|<1.\] The left side of \eqref{Ordinary_Partitions_Analytic_Identity} is not manifestly non-negative due to the $1/(1+q^n)$ factors appearing in the summands, but the series on the right side (which is related to the $t(\pi)$ statistics) shows the non-negativity, as expected.

In this work, we introduce a refinement of Theorem~\ref{Ordinary_Partitions_Combinatorial_Weighted_Theorem} where we put a bound on the difference between the largest and the smallest parts of partitions. We prove that 
\begin{theorem}\label{Smallest_part_positivity_THM} For $L\geq 1$,
\begin{equation}\label{weighted_sum_of_bounded_difference}
\sum_{\substack{\pi\in\mathcal{U}\\l(\pi)-s(\pi)\leq L}} (-1)^{s(\pi)+1}q^{|\pi|} = \sum_{s\geq 1} \frac{(-1)^{s+1}q^s}{(q^s;q)_{L+1}}\succcurlyeq 0.
\end{equation}
\end{theorem} 

We remark that for $L=0$, the right-hand side becomes \[\sum_{s\geq 1} \frac{(-1)^{s+1}q^s}{1-q^s}.\] Although Theorem~\ref{Smallest_part_positivity_THM} does not apply for this case. It is easy to conclude that \[\sum_{\substack{\pi\in\mathcal{U}\\l(\pi)=s(\pi)\\|\pi|=N}} (-1)^{s(\pi)+1}q^{|\pi|} \geq 0, \hspace{.5cm}\text{if}\hspace{.4cm} 4\nmid N.\]

Interested readers are invited to examine  \cite{AndrewsBeckRobbins}, \cite{Chapman}, and \cite{Kronholm} for other studies on bounded differences between largest and smallest parts.

Section~\ref{Section_Reper} has a short repertoire of basic hypergeometric identities that will be referred to later. In Section~\ref{Section_Ineq} we are going to prove two inequalities between sets of partitions (Theorem~\ref{Ali_comb_THM} and an analogue) using only injections between sets, and later state some related open questions. We will state the analytic versions of some of the theorems of Section~\ref{Section_Ineq} and their implications in Section~\ref{Section_non_negativity}. We later will use the complements of the range of the injective maps of Section~\ref{Section_Ineq} to get new $q$-series summation formulas. Theorem~\ref{Smallest_part_positivity_THM} will be proven in Section~\ref{Section_Proof_of_First_theorem}. Section~\ref{Section_Refinement} has an excursion in different representations and an observably non-negative expression for the analytic expression of \eqref{weighted_sum_of_bounded_difference} of Theorem~\ref{Smallest_part_positivity_THM}. An outlook section finishes the paper with a summary of open questions that arise from this study.

\section{Some $q$-Hypergeometric Identities}\label{Section_Reper}
Some $q$-hypergeometric functions and some of their related formulas that will be used later are stated here. Let $r$ and $s$ be non-negative integers and $a_1,a_2,\dots,a_r,b_1,b_2,\dots,b_s,q,$ and $z$ be variables. Then, \begin{equation}\label{r_Phi_s}_r\phi_s\left(\genfrac{}{}{0pt}{}{a_1,a_2,\dots,a_r}{b_1,b_2,\dots,b_s};q,z\right):=\sum_{n=0}^\infty \frac{(a_1;q)_n(a_2;q)_n\dots (a_r;q)_n}{(q;q)_n(b_1;q)_n\dots(b_s;q)_n}\left[(-1)^nq^{n\choose 2}\right]^{1-r+s}z^n.\end{equation}
Let $a$, $b$, $c$, $q$, and $z$ be variables. The $q$-binomial theorem \cite[II.3, p. 236]{GasperRahman} is \begin{equation}\label{q_binomial}
{}_1\phi_0 \left(\genfrac{}{}{0pt}{}{a}{-};q,z \right) = \frac{(az;q)_\infty}{(z;q)_\infty}.
\end{equation} All three Heine transformations \cite[III.1-3, p. 241]{GasperRahman}
\begin{align}\label{Heine1} are
{}_2\phi_1 \left(\genfrac{}{}{0pt}{}{a,\ b}{c};q,z \right) &= \frac{(b;q)_\infty(az;q)_\infty}{(c;q)_\infty(z;q)_\infty} {}_2\phi_1 \left(\genfrac{}{}{0pt}{}{c/b,\ z}{az};q,b \right),  \\ \label{Heine2} &= \frac{(c/b;q)_\infty (bz;q)_\infty}{(c;q)_\infty(z;q)_\infty} {}_2\phi_1 \left(\genfrac{}{}{0pt}{}{abz/c,\ b}{bz};q,c/b \right), \\ \label{Heine3} &= \frac{(abz/c;q)_\infty}{(z;q)_\infty}{}_2\phi_1 \left(\genfrac{}{}{0pt}{}{c/a,\ c/b}{c};q,abz/c \right) .
\end{align} The Jackson transformation \cite[III.4, p. 236]{GasperRahman} is \begin{equation}
\label{Jackson_transformation}{}_2\phi_1 \left(\genfrac{}{}{0pt}{}{a,\ b}{c};q,z \right)  = \frac{(az;q)_\infty}{(z;q)_\infty}{}_2\phi_2 \left(\genfrac{}{}{0pt}{}{a,\ c/b}{c,\ az};q,bz \right).\end{equation}

\section{Two New Partition Inequalities}\label{Section_Ineq}

We start our discussion with a proof of Theorem~\ref{Ali_comb_THM}. Recall:\\

\noindent\textbf{Theorem~\ref{Ali_comb_THM}.} For any $L\geq 2$ and $N\geq 1 $,
\begin{equation*}
|\{\pi:\pi\in\mathcal{A}_{L,1},\ |\pi|=N\}|\geq|\{\pi:\pi\in\mathcal{A}_{L,2},\ |\pi|=N\}|.
\end{equation*}

Note that the claimed inequality of Theorem~\ref{Ali_comb_THM} is not true for $L=1$ as the set $A_{1,1}$ only has partitions of type $(1^1,2^{f_2})$ and $A_{1,2}$ only has partitions of type $(2^{1+f_2})$. Hence, $\mathcal{A}_{1,1}$ has a single partition for every odd norm and $\mathcal{A}_{1,2}$ has a single partition for every even norm and nothing else, making the inequality claim of \eqref{Ali_comb_inequality} impossible for this case.

For $L\geq 2$, we prove the inequality \eqref{Ali_comb_inequality} in an injective manner. 

\begin{proof}
First we handle the $L=2$ case with the injection $\gamma^*: \mathcal{A}_{2,2}\rightarrow \mathcal{A}_{2,1}$. Let $\pi=(2^{f_2},3^{f_3})\in \mathcal{A}_{2,2}$, then \begin{enumerate}[i.]
\item if  $f_2>0$, then $\gamma^*(\pi) = (1^{2f_2},2^0,3^{f_3})$,\\[-1.5ex]
\item if $f_2=0$ and $f_3>0$, then $\gamma^*(\pi) = (1^{3},2^0,3^{f_3-1})$.
\end{enumerate}
The parity of the frequency of 1 in the image clearly determines the case. Hence, $\gamma^*$ is an injection demonstrating \eqref{Ali_comb_inequality} for $L=2$.

For $L\geq 3$, let $\pi = (1^0,\ 2^{f_2},\ \dots,\ L^{f_L},\ (L+1)^{f_{L+1}})$ be a partition from the set $\mathcal{A}_{L,2}$. Define $\gamma~:~\mathcal{A}_{L,2}~\rightarrow~\mathcal{A}_{L,1}$ by the following cases. 

\begin{enumerate}[i.]
\item If $2< s(\pi) <L+1$, then $\gamma(\pi) = (1^{[(f_L-\delta_{L,s(\pi)})\cdot L + 1]},(s(\pi)-1)^{1},s(\pi)^{f_{s(\pi)}-1},\dots, L^{0},(L+1)^{f_{L+1}})$,\\[-1.5ex]

\item if $s(\pi) = L+1$, then $\gamma(\pi) = (1^{L+1},(L+1)^{f_{L+1}-1})$,\\[-1.5ex]

\item if $s(\pi) = 2$, then $\gamma(\pi) = (1^{(f_L\cdot L + 2)},2^{f_2-1},\dots, L^{0},(L+1)^{f_{L+1}})$.
\end{enumerate}

The image of a partition $\pi \in \mathcal{A}_{L,2}$ is uniquely defined. The remainder of the frequency $f_1$ divided by $L$ in the image is either 1 or 2. The remainder 2 comes from a unique case. In the remainder being 1 cases, one can uniquely identify the pre-image by looking at the smallest part size that is greater than 1. This proves that $\gamma$ is an injection and it is enough to show \eqref{Ali_comb_inequality}.
\end{proof}

Interested reader is invited to examine \cite{AntiTelescoping}, \cite{BerkovichGarvan_dissecting}, \cite{BerkovichGrizzell1}, \cite{BerkovichGrizzell2}, \cite{BerkovichGrizzell3}, and \cite{McLaughlin} for other examples of injective combinatorial arguments and inequalities between the sizes of sets of partitions. 

We exemplify Theorem~\ref{Ali_comb_THM} with Table~\ref{Table_Ali_comb} by writing out the related partitions.

\begin{table}[h]\caption{Example of Theorem~\ref{Ali_comb_THM} with $L=3$ and $N=12$, where the images of the map $\gamma$ are also indicated.}\label{Table_Ali_comb}
\begin{tabular}{ccc}
$\begin{array}{c}
\pi\in\mathcal{A}_{3,2}\\[-1.8ex] \\\\[-1.8ex]\\ (3^4)\\[-1.8ex]\\(2,3^2,4)\\[-1.8ex] \\(2^3,3^2)\\[-1.8ex]\\\\[-1.8ex] \\\\[-1.8ex] \\ (4^3)\\[-1.8ex]\\\\[-1.8ex]\\\\[-1.8ex] \\ (2^2,4^2)\\[-1.8ex]\\ (2^4,4) \\[-1.8ex]\\(2^6)\\[-1.8ex]
\end{array}
$
&
$\begin{array}{c}
\gamma \\[-3ex]\\\\[-1.8ex]\\\\[-1.8ex] \rightarrow\\[-1.8ex]\\\rightarrow \\[-1.8ex]\\\rightarrow\\[-1.8ex]\\\\[-1.8ex] \\ \\[-1.8ex]\\ \rightarrow\\[-1.8ex]\\\\[-1.8ex]\\\\[-1.8ex]\\ \rightarrow\\[-1.8ex]\\ \rightarrow\\[-1.8ex] \\\rightarrow\\[-1.8ex]
\end{array}$
 &
$\begin{array}{c}
\pi\in\mathcal{A}_{3,1}\\[-1.8ex] \\(1^{12})\\[-1.8ex]\\ (1^{10},2)\\[-1.8ex]\\ (1^8,4)\\[-1.8ex]\\(1^8,2^2)\\[-1.8ex]\\ (1^6,2^3)\\[-1.8ex]\\ (1^6,2,4)\\[-1.8ex]\\ (1^4,4^2)\\[-1.8ex]\\(1^4,2^2, 4)\\[-1.8ex]\\ (1^4, 2^4)\\[-1.8ex]\\ (1^2,2,4^2)\\[-1.8ex]\\ (1^2,2^3,4)\\[-1.8ex] \\(1^2,2^5)\\[-1.8ex]
\end{array}$
\end{tabular}
\end{table}

We can shift the permissible parts of the sets $\mathcal{A}_{L,i}$ up by one and also get a similar result to Theorem~\ref{Ali_comb_THM}. Let $L\geq 1$ be an integer and define \begin{enumerate}[i.]
\item $\mathcal{B}_{L,1}$ be the set of partitions with the smallest part is 2, all the parts are $\leq L+2$ and $f_{L+1} = \delta_{L,1}$,\\[-1.5ex]
\item$\mathcal{B}_{L,2}$ be the set of non-empty partitions where the parts are in the domain $\{3,4,\dots,L+2 \}$.
\end{enumerate}Then we have 

\begin{theorem}\label{Ali_comb_THM2} For any $L\geq 3$ and $N\geq 1 $
\begin{equation}\label{Ali_comb_inequality2}|\{\pi:\pi\in\mathcal{B}_{L,1},\ |\pi|=N\}| + \delta_{N,3} + \delta_{N,9}\delta_{L,4}\ \geq\ |\{\pi:\pi\in\mathcal{B}_{L,2},\ |\pi|=N\}|.\end{equation}
\end{theorem}

Before the proof of Theorem~\ref{Ali_comb_THM2}, we examine the excluded initial cases of $L$. In the case $L=1$, $\mathcal{B}_{1,1}$ is the set of partitions of type $(2^1,3^{f_3})$. Hence, all partitions of $\mathcal{B}_{1,1}$ have norm 2 modulo 3. The set $\mathcal{B}_{1,2}$ contains partitions only of the type $(3^{1+f_3})$, which 0 modulo 3 norm. Therefore, the inequality \eqref{Ali_comb_inequality2} cannot hold for all $N$. The sets $\mathcal{B}_{2,1}$ and $\mathcal{B}_{2,2}$ contain partitions exclusively of the type $(2^{1+f_2},4^{f_4})$ and $(3^{f_3},4^{f_4})$ with $f_3+f_4>0$, respectively. It is easy to see that all the partitions in $\mathcal{B}_{2,1}$ have even norms, but for any $k\geq 0$, there are partitions of norm $4k+3$ in $\mathcal{B}_{2,2}$. Therefore, for $L=2$, the inequality \eqref{Ali_comb_inequality2} does not hold for all $N$ either.

\begin{proof}
We begin our proof with the $L=3$ case. Let $\pi = (3^{f_3},4^{f_4},5^{f_5})$ be a partition in $\mathcal{B}_{3,2}$, with norm $> 3$. Let $\Gamma_1^*$ be the map from $\mathcal{B}_{L,2}$ to $\mathcal{B}_{L,1}$ as follows:
\begin{enumerate}[i.]
\item If $f_4 >0$, then $\pi\mapsto (2^{2f_4}, 3^{f_3},5^{f_5})$, \\[-1.5ex]
\item if $f_3=f_4=0$, then $f_5>0$ and define $\pi\mapsto (2^1,3^1,5^{f_5-1})$,\\[-1.5ex]
\item if $f_4=0$ and $f_3>1$, then $\pi\mapsto (2^3,3^{f_3 -2},5^{f_5})$,\\[-1.5ex]
\item if $f_4=0$ and $f_3=1$, since $|\pi|>3$, $f_5>0$, then $\pi\mapsto (2^1,3^{2},5^{f_5-1})$.
\end{enumerate}
This case by case map $\Gamma_1^*$ can easily be seen to be an injection. In cases i. and iii. the frequency of 2 as a part in the image, is the signature, and in the other cases frequency of 3 becomes our signature. This distinguishes all the cases from each other. 

Let $\pi = (3^{f_3},\dots,(L+2)^{f_{L+2}})$ be a partition in $\mathcal{B}_{L,2}$, with norm $> 3$. For $L=2m-1>5$, we define the injective map $\Gamma_1$ as follows:
\begin{enumerate}[i.]
\item If $f_{L+1}=f_{2m}>0$, then $\pi\mapsto (2^{f_{2m}\cdot m},3^{f_3},\dots, (L+1)^0, (L+2)^{f_{L+2}})$,\\[-1.5ex]
\item if $\exists i\in\{2,\dots, m-1\}$ such that $f_{2i}>0$ and $f_{2j}=0$, $\forall j> i$, then \[\pi\mapsto (2^{i},3^{f_3},\dots,(2i)^{f_{2i}-1},(2i+1)^{f_{2i+1}},(2i+2)^0,\dots,(L+1)^0, (L+2)^{f_{L+2}}),\]
\item if $\forall i \in\{2,\dots, m\}$, $f_{2i}=0$ and $s(\pi)$ is odd $> 3$, then \[\pi\mapsto (2^{1},(s(\pi)-2)^1,(s(\pi)-1)^0,s(\pi)^{f_{s(\pi)}-1},\dots),\]
\item if $\forall i \in\{2,\dots, m\}$, $f_{2i}=0$ and $f_3\geq2$, then $\pi\mapsto (2^{1},3^{f_3-2},4^1,5^{f_5},\dots)$,\\[-1.5ex]
\item if $\forall i \in\{2,\dots, m\}$, $f_{2i}=0$ and $f_3=1$, then since $|\pi|>3$ there is a smallest positive $j>1$ such that $f_{2j+1}>0$, then $\pi\mapsto (2^{1},(j+1)^2,(2j+1)^{f_{2j+1}-1},\dots)$.
\end{enumerate}

The map $\Gamma_1$ for odd $L \geq 5$ is injective as the number of occurrences of $2$, if larger than $1$, specifies the case and if $2$ appears only once in the image then the following smallest parts specify the case.

One can also view $\Gamma_1^*$ for $L=3$ as a derivation of $\Gamma_1$. We use the cases i., iii. and v. of $\Gamma_1$ as is and modify the iv. as $\pi \mapsto (2^3,3^{f_3-2},5^{f_5})$. The case ii. of $\Gamma_1$ does not apply for $L=3$.

Now we define the map $\Gamma_2^*$ for $L=4$. Let $\pi = (3^{f_3},4^{f_4},5^{f_5},6^{f_6})$ with $|\pi|\not=3$.
Then the map $\Gamma_2^*$ sends $\pi$ to the following images depending on the following cases.
\begin{enumerate}[i.]
\item If $f_5$ is positive even, then $\pi\mapsto (2^{(f_5/2)5},3^{f_3},4^{f_4},5^{0},6^{f_6})$,\\[-1.5ex]
\item if $f_5$ is positive odd and if $f_3>0$, then $\pi\mapsto (2^{((f_5-1)/2)5+4},3^{f_3-1},4^{f_4},5^{0},6^{f_6})$,\\[-1.5ex]
\item if $f_5$ is positive odd and if $f_3=0$, then $\pi\mapsto (2^{((f_5-1)/2)5+1},3^{1},4^{f_4},5^{0},6^{f_6}),$\\[-1.5ex]
\item if $f_5=0$,\\[-1.5ex] \begin{enumerate}[1.] \item and $f_6>0$, then $\pi\mapsto (2^{3},3^{1},4^{f_4},5^{0},6^{f_6-1})$,\\[-1.5ex]
\item or $f_6=0$ and $f_4>0$, then $\pi\mapsto (2^{2},3^{1},4^{f_4-1},5^{0},6^{0})$,\\[-1.5ex]
\end{enumerate} 
\item if $f_4=f_5=f_6=0$, since $|\pi|\not= 3$,\\[-1.5ex] \begin{enumerate}[1.] \item either $f_3=2$, then $\pi\mapsto (2^{1},3^{0},4^{1},5^{0},6^{0})$,\\[-1.5ex]
\item or $f_3 \geq 3$, then $\pi\mapsto (2^{1},3^{f_3-2},4^{1},5^{0},6^{0})$.\\[-1.5ex]
\end{enumerate} 
\end{enumerate}
The partitions $(4,5)$ and $(3^3)$ in $\mathcal{B}_{4,2}$ both get mapped to $(2,4,5)$, which is a source of the extra correction term of size 1, for $N=9$. Other than this explained issue, all the images of $\Gamma_2^*$ can easily be classified and the inverse images can be found by looking at the frequency of 2 modulo 5. If the frequency of $2$ is exactly one, then the frequency of 3 determines the case and sub-case the image is coming from. This injective map can be generalized for larger even $L$. We define $\Gamma_2$ for all even $L= 2m \geq 6$. Let $\pi = (3^{f_3},\dots,(L+2)^{f_{L+2}})$ be a partition in $\mathcal{B}_{L,2}$, with norm $> 3$.
\begin{enumerate}[i.]
\item If $f_{L+1}=f_{2m+1}$ is positive even, then $\pi\mapsto (2^{(f_{2m+1}/2)(2m+1)},3^{f_3},\dots, (L+1)^0, (L+2)^{f_{L+2}})$,\\[-1.5ex]
\item if $f_{2m+1}$ is positive odd and if $\exists k \in\{2,\dots,m\}$ with $f_{2k-1}>0$ where $\forall k<j<m+1$, $f_{2j-1}=0$, then \[\pi\mapsto (2^{[(f_{2m+1}-1)/2](2m+1) + (m+k)},\dots,(2k-1)^{f_{2k-1}-1},\dots,(L+1)^0, (L+2)^{f_{L+2}}),\]
\item if $f_{2m+1}$ is positive odd and $\forall k \in\{2,\dots,m\}$, $f_{2k-1}=0$, then \[\pi\mapsto (2^{[(f_{2m+1}-1)/2](2m+1) + 1}, \dots,(2m-1)^{1},(2m)^{f_{2m}},(L+1)^0, (L+2)^{f_{L+2}}),\]
\item if $f_{2m+1}=0$, and there exist largest $k \in\{2,\dots,m +1\}$ such that $f_{2k} >0$, then \[\pi\mapsto (2^{k},3^{f_3}, \dots,(2k)^{f_{2k}-1},\dots),\]
\item if $f_{2m+1}=0$, and $\forall k \in\{2,\dots,m +1\}$, $f_{2k} = 0$ and if $f_3=1$, then since $|\pi|>3$ there exists the smallest positive $i>1$ such that $f_{2i+1}>0$, then $\pi\mapsto (2^{1},(i+1)^2,(2i+1)^{f_{2i+1}-1},\dots)$,\\[-1.5ex]
\item if $f_{2m+1}=0$, and $\forall k \in\{2,\dots,m +1\}$ such that $f_{2k} = 0$ and if $f_3>1$, then $\pi\mapsto (2^{1},3^{f_3-2},4^1,\dots)$,\\[-1.5ex]
\item if $f_{2m+1}=0$, and $\forall k \in\{2,\dots,m +1\}$ such that $f_{2k} = 0$ and if $f_{3}=0$ then there exists smallest integer $m>i>1$ such that $f_{2i+1}>0$, then $\pi\mapsto (2^{1},(2i-1)^1,(2i+1)^{f_{2i+1}-1},\dots)$,
\end{enumerate}

The $\Gamma_2$ injection, just like $\Gamma_1$, has no problem in separating the cases when $f_2 \not = 1$ in the image of partitions. The $f_2=1$ cases in the image can be identified uniquely by the second and third smallest parts and their frequencies. The condition $L \geq 6$ or equivalently $m\geq 3$ is used implicitly as it is necessary for the vii. case to be defined.

For the $L=4$ case our injection $\Gamma_2^*$ for partitions $|\pi|$ with norm not equal to $3$ or $9$ can be related with $\Gamma_2$, where we use the cases i.--iv. and vi. with $m=2$, where the case vi. comes with the extra assertion that $f_3-2 \not = 1$. 
\end{proof}

An example of Theorem~\ref{Ali_comb_THM2} is given in Table~\ref{Table_Ali_comb2}.

\begin{table}[ht]\caption{Example of Theorem~\ref{Ali_comb_THM2} with $L=5$ and $N=12$, where the images of the map $\Gamma_1$ are also indicated.}\label{Table_Ali_comb2}
\begin{tabular}{ccc}
$\begin{array}{c}
\pi\in\mathcal{B}_{5,2} \\ (6^2) \\(3^2,6) \\ (3^4) \\(3,4,5) \\(5,7) \\  \\ (4^3)\\ {\color{white}a}
\end{array}
$
&
$\begin{array}{c}
\Gamma_1 \\ \rightarrow \\ \rightarrow \\ \rightarrow  \\ \rightarrow \\ \rightarrow \\\\ \rightarrow\\ {\color{white}a}
\end{array}$
 &
$\begin{array}{c}
\pi\in\mathcal{B}_{5,1} \\(2^6) \\ (2^3,3^2) \\  (2,3^2,4)\\ (2^2,3,5) \\ (2,3,7) \\ (2^4,4)  \\ (2^2,4^2) \\ (2,5^2)
\end{array}$
\end{tabular}
\end{table}

Theorem~\ref{Ali_comb_THM} and \ref{Ali_comb_THM2} are intriguing and can also be viewed as the initial stages of a more general conjecture. Define the following sets \begin{enumerate}[i.]
\item $\mathcal{C}_{L,s,1}$ denotes the set of partitions where the smallest part is $s$, all the parts are $\leq L+s$ and $L+s-1$ doesn't appear as a part,\\[-1.5ex]
\item$\mathcal{C}_{L,s,2}$ denotes the set of non-empty partitions where the parts are in the domain $\{s+1,\dots,L+s \}$.
\end{enumerate}

\begin{conjecture}\label{Sets_Conjecture} For given integers $L$ and $s$ there exist $M$, which depends on $s$ only, such that
\begin{equation}\label{Ali_comb_conjecture3}|\{\pi:\pi\in\mathcal{C}_{L,s,1},\ |\pi|=N\}| \ \geq\ |\{\pi:\pi\in\mathcal{C}_{L,s,2},\ |\pi|=N\}|,\end{equation} for all $N\geq M$.
\end{conjecture}

The first two initial families of cases for $s=1$ and $2$ are Theorem~\ref{Ali_comb_THM} and \ref{Ali_comb_THM2} with $M=1$ and $M=10$. It should be noted that in a case when $L$ tends to $\infty$, this conjecture is nothing but a tautology. 

In the definition of $\mathcal{B}_{L,i}$, we shifted the permissible part sizes of $\mathcal{A}_{L,i}$ up by one. Another route to take would be shifting the sets, but keeping the impermissible part $L$ of $\mathcal{A}_{L,i}$ the same. For $L\geq s+1$, let
\begin{enumerate}[i.]
\item $\mathcal{C}^*_{L,s,1}$ be the set of partitions where the smallest part is $s$, all the parts are $\leq L+s$ and $L$ doesn't appear as a part,\\[-1.5ex]
\end{enumerate}

Similar to Conjecture~\ref{Sets_Conjecture} we also claim that

\begin{conjecture}\label{Sets_Conjecture2} For given integers $L$ and $s$ there exist $M$, which only depends on $s$, such that
\begin{equation}\label{Ali_comb_conjecture4}|\{\pi:\pi\in\mathcal{C}^*_{L,s,1},\ |\pi|=N\}| \ \geq\ |\{\pi:\pi\in\mathcal{C}_{L,s,2},\ |\pi|=N\}|,\end{equation} for all $N\geq M$.
\end{conjecture}

In Section~\ref{Section_Outlook} we will be reiterating these conjectures and state their analytic versions.

\section{Some Analytic Non-negativity Results and Alternative Representations}\label{Section_non_negativity}

Theorem~\ref{Ali_comb_THM} and Theorem~\ref{Ali_comb_THM2} lead to new non-negativity results and some new summation formulas. The analytic analogue of Theorem~\ref{Ali_comb_THM} is the following:

\begin{theorem}\label{H_positivity_theorem} For $L\geq 2$,  \begin{equation}\label{H_original}H_{L,1}(q):= \frac{q}{(q;q)_{L-1}(1-q^{L+1})} - \left(\frac{1}{(q^2;q)_L} - 1\right) \succcurlyeq 0.\end{equation} 
\end{theorem}

We can, and will, extend the definition of $H_{L,1}(q)$ for $L=1$ case, but in this case, the expression simplifies to $q/(1+q)$ and is not non-negative.

\begin{proof}
All we need to point out is that \[\frac{q}{(q;q)_{L-1}(1-q^{L+1})}\] is the generating function for the number of partitions from the set $\mathcal{A}_{L,1}$ and that\[\frac{1}{(q^2;q)_L} - 1\] is the generating function for the number of partitions from the set $\mathcal{A}_{L,2}$. Theorem~\ref{Ali_comb_THM} proves the non-negativity assertion for $H_{L,1}(q)$, where $L\geq 2$.
\end{proof}

On that note any expression of the following form is non-negative \begin{equation}\label{non_neg_products}\prod_{i \in \mathcal{I}}\frac{1}{(1-q^i)} -1 \succcurlyeq 0, \end{equation} for any $\mathcal{I}\subset \mathbb{N}$. The first term can be thought as the generating function for the partitions with parts in the set $\mathcal{I}$ and the $-1$ term can be interpreted as taking away the empty partition from the calculations. This type of \textit{``reciprocal product take away one"} expressions of the form \eqref{non_neg_products} will appear in our future calculations, and they are always going to be non-negative by this observation.

The coefficients of $H_{1,1}(q) = q/(1+q)$ are consistent with our earlier observations about the $\mathcal{A}_{1,1}$ and $\mathcal{A}_{1,2}$, which came immediately before the proof of Theorem~\ref{Ali_comb_THM}. In general, we can write $H_{L,1}(q)$ abstractly, as the difference of generating functions. For any non-negative $L$, \[H_{L,1}(q) = \sum_{\pi\in\mathcal{A}_{L,1}}q^{|\pi|} - \sum_{\pi\in\mathcal{A}_{L,2}}q^{|\pi|}. \] 

Now that we know the relation between the injections $\gamma^*$, and $\gamma$ defined in the proof of Theorem~\ref{Ali_comb_THM} and $H_{L,1}(q)$, we can find an alternative expression for $H_{L,1}(q)$ by considering the elements of $\mathcal{A}_{L,1}$ that are not an image of these injections. For $L=2$, it is clear that the partitions of type \begin{enumerate}[i.]
\item $(1^1,2^0,3^{f_3})$, where $f_3$ is a non-negative integer,\\[-1.5ex] 
\item $(q^{2j+5},2^0,3^{f_3})$ for $j$ and $f_3$ non-negative integers,
\end{enumerate} 
are the elements of $\mathcal{A}_{2,1}$, which are not in the range of $\gamma^*$. The generating functions of such partitions can be easily written as 
\begin{equation}\label{H_sec_1}\frac{q}{1-q^3} + \frac{q^5}{(1-q^2)(1-q^3)}.\end{equation}

Let $L\geq 3$, given partition $\pi\in\mathcal{A}_{L,1} \backslash \gamma(\mathcal{A}_{L,2})$, $\pi$ can have one of the following three forms: 
\begin{enumerate}[i.]
\item For $2 < s < L+1$, $(1^{tL+1},(s-1)^{k+2},\dots,L^0,(L+1)^{f_{L+1}})$, where $t$ and $k$ are non-negative integers,\\[-1.5ex]
\item $(1^{kL+1},L^0,(L+1)^{f_{L+1}})$, where $k \geq 2$ or $0$,\\[-1.5ex]
\item $(1^{kL+r},\dots,L^0,(L+1)^{f_{L+1}})$, where $r\in\{3,\dots,L\}$ and $k$ is a non-negative integer.
\end{enumerate}

The generating functions for these cases are given by the first term, the following two, and the last term in the following expression for $L\geq 3$, respectively.

\begin{equation}\label{H_sec_2}\sum_{s=2}^{L-1} \frac{q^{2s+1}}{(q^s;q)_{L+2-s}} + \frac{q^{2L+1}}{(1-q^L)(1-q^{L+1})} + \frac{q}{1-q^{L+1}} + \frac{q^3(1+q+\dots+q^{L-3})}{(q^2;q)_{L}}.\end{equation}

This alternative formula \eqref{H_sec_2} can be shortened a little by combining the first two terms and rewriting the last.

\begin{equation}\label{H_secondary}H_{L,1}(q) = \sum_{s=2}^{L} \frac{q^{2s+1}}{(q^s;q)_{L+2-s}}  + \frac{q}{1-q^{L+1}} + \frac{q^3(1-q^{L-2})}{(q;q)_{L+1}}. \end{equation}

One important note about \eqref{H_secondary} is that it is written with manifestly non-negative terms. In fact, this formula can be checked to be valid for $L=1$ and $2$ (consistent with \eqref{H_sec_1}\hspace{.5mm}) as well, even though the $\gamma$ map is not defined for these cases. 

Equating \eqref{H_original} and \eqref{H_secondary} yields the formula by using combinatorial means only. 
\begin{theorem}\label{summation_formula_F_recurr} For a positive integer $L$,
\begin{equation}\label{H_finite_summation_formula}\sum_{s=1}^{L} \frac{q^{2s+1}}{(q^s;q)_{L+2-s}} =1- \frac{q}{1-q^{L+1}}+\frac{2q-1}{(q;q)_{L+1}} .\end{equation}
\end{theorem}

A direct proof can also be given.

\begin{proof} We start by noting \[\sum_{s=1}^{L} \frac{q^{2s+1}}{(q^s;q)_{L+2-s}} = \frac{q^3}{(q;q)_{L+1}}\sum_{s=0}^{L-1} q^{2s}(q;q)_{s} .\]
Observe that \[\sum_{i=0}^{L-1} q^{i+1}(q;q)_{i} = 1-(q;q)_L.\] This is because the left-hand side sum is the generating function for the number of non-zero partitions $\pi$ into distinct parts $\leq L$ with weights $(-1)^{\nu(\pi)-1}$ written with respect to the largest part of the partitions. Now, it is easy to justify
\[q^2\sum_{i=0}^{L-1} q^{2i}(q;q)_{i} = 1-(1-q)\sum_{i=0}^{L-1} q^{i}(q;q)_{i} -q^L (q;q)_L. \]

By dividing both sides by $q/(q;q)_{L+1}$, and doing the necessary simplifications, one can finish the proof.
\end{proof}

Taking the limit $L\rightarrow\infty$ in \eqref{H_finite_summation_formula}, it is easy to get:

\begin{corollary}\label{different_types_of_H_sub_infty_theorem}\begin{equation}\label{q3_F2_open_formula}
\sum_{s\geq1} \frac{q^{2s+1}}{(q^s;q)_{\infty}} =1-q+\frac{2q-1}{(q;q)_\infty}.
\end{equation}
\end{corollary}


One can also give a direct $q$-hypergeometric proof of Corollary~\ref{different_types_of_H_sub_infty_theorem}. This proof amounts to using the second Heine transformation \eqref{Heine2} followed by the  $q$-binomial theorem \eqref{q_binomial}.



Similar to Theorem~\ref{Ali_comb_THM}, Theorem~\ref{Ali_comb_THM2} also has a $q$-theoretic equivalent.

\begin{theorem}\label{H_positivity_theorem2} For $L\geq 3$,  \begin{equation}\label{H_2_formula} H_{L,2}(q):=q^3 + \delta_{L,4}\ q^9 + \frac{q^2(1-q^{L+1})}{(q^2;q)_{L+1}} - \left(\frac{1}{(q^3;q)_{L}} - 1\right)\succcurlyeq 0.\end{equation}
\end{theorem}

\begin{proof} For $L\geq 3$, the generating functions for the number of partitions coming from the sets $\mathcal{B}_{L,1}$ and $\mathcal{B}_{L,2}$ are \[ \frac{q^2(1-q^{L+1})}{(q^2;q)_{L+1}} \text{  and  } \frac{1}{(q^3;q)_{L}} - 1,\] respectively. The correction terms for norm 3 and the one time correction term for norm 9 cases are also added analogous to \eqref{Ali_comb_inequality2}. Theorem~\ref{Ali_comb_THM2} proves the claimed non-negativity.
\end{proof}

The $H_{L,2}(q)$ can be extended to the positive integers and, in general, can be written as a difference of generating functions with two extra factors as \[H_{L,2}(q)= q^3 + \delta_{L,4}q^9 +\sum_{\pi\in\mathcal{B}_{L,1}} q^{|\pi|}- \sum_{\pi\in\mathcal{B}_{L,2}} q^{|\pi|}. \] The initial cases of $H_{L,2}$ are as follows \begin{align*}
H_{1,2}(q) &= q^3 + \frac{q^2}{1-q^3}-\frac{q^3}{1-q^3}=q^2+q^5-q^6+\dots,\\
H_{2,2}(q) &= q^3 + \frac{q^2}{(1-q^2)(1-q^4)} - \left( \frac{1}{(1-q^3)(1-q^4)}-1 \right)=q^2+q^6-q^7+\dots,
\end{align*} which clearly show that Theorem~\ref{H_positivity_theorem2} fails for $L=1$ and 2.

Similar to the treatment of the injection $\gamma$ after Theorem~\ref{H_positivity_theorem}, one can look for the partitions that are outside of the image of $\Gamma^*_1$, $\Gamma_1$ etc. and write the \eqref{H_2_formula} expression with manifestly non-negative terms. Yet, the increase in the number of cases are making this study not necessarily harder, but messier.

Considering partitions outside of the image of $\Gamma^*_1$, defined in the proof of Theorem~\ref{Ali_comb_THM2},  for $L=3$ case implies the summation formula:

\begin{equation}\label{H_32_secondary} H_{3,2}(q) = \frac{q^{10}}{(q^3;q)_3} + \frac{q^{11}}{(q^3;q^2)_2} + \frac{q^2}{(1-q^5)},\end{equation}

where the partitions in $\mathcal{B}_{3,1} \backslash \Gamma^*_1(\mathcal{B}_{3,2})$ are of the form\begin{enumerate}[i.] 
\item $(2^{2j+5},3^{f_3},5^{f_5})$, where $j$ is a non-negative integer,\\[-1.5ex]
\item or $(2^1,3^{f_3},5^{f_5})$, where $f_3 \geq 3$,\\[-1.5ex]
\item or $(2^1,5^{f_5})$,
\end{enumerate}
according to the three summation terms on the right-hand side of \eqref{H_32_secondary}. Recall that $f_i$ is non-negative for any $i$. For larger odd $L$ values that fall under the injective map $\Gamma_1$, we can repeat this process and write $H_{L,1}$ as a sum of manifestly non-negative terms. Needless to say, $f_{L+1}=0$ in these cases. The partitions in $\mathcal{B}_{L,1} \backslash \Gamma_1(\mathcal{B}_{L,2})$, for odd $L>3$, are ones of the form:
\begin{enumerate}[i.] 
\item $(2^{i+k(L+1)/2},3^{f_3},\dots,(L+1)^0, (L+2)^{f_{L+2}})$, where $(L-1)/2\geq i \geq 2$ and $k~+~\sum_{j=i+1}^{(L-1)/2}~f_{2j}~>~0$,\\[-1.5ex]
\item $(2^{1+k(L+1)/2},(L+2)^{f_{L+2}})$, where $k$ is a non-negative integer,\\[-1.5ex]
\item $(2^{1+k(L+1)/2}, s^{3+f_s},\dots,(L+1)^0 ,(L+2)^{f_{L+2}})$, where $L\geq s\geq 4$,\\[-1.5ex]
\item $(2^{1+k(L+1)/2}, s^2,\dots,(L+1)^0 ,(L+2)^{f_{L+2}})$, where $L \geq s\geq (L+5)/2$,\\[-1.5ex]
\item $(2^{1+k(L+1)/2}, s^2,\dots,(L+1)^0 ,(L+2)^{f_{L+2}})$, where $(L+3)/2 \geq s\geq 4$ and\newline $k + \sum_{j=s+1}^{2s-2}f_j+\sum_{j=s}^{(L-1)/2}f_{2j}>0$,\\[-1.5ex]
\item $(2^{1+k(L+1)/2}, s^1,\dots, (L+1)^0,(L+2)^{f_{L+2}})$, where $s$ is odd and $L\geq s\geq 5$, $k+\sum_{j=(s+1)/2}^{(L-1)/2} f_j>0$, \\[-1.5ex]
\item $(2^{1+k(L+1)/2}, s^1,\dots,(L+1)^0,(L+2)^{f_{L+2}})$, where $(L-1)/2\geq s>4$ even,\\[-1.5ex]
\item $(2^{1+k(L+1)/2},4^1,\dots,(L+1)^0,(L+2)^{f_{L+2}})$, where $k+\sum_{j=3}^{(L-1)/2}f_{2j}>0$,\\[-1.5ex]
\item $(2^{1+k(L+1)/2},3^{1+f_3},4^{2+f_4},\dots,(L+1)^0,(L+2)^{f_{L+2}})$,
\item $(2^{1+k(L+1)/2},3^{1+f_3},4^{\alpha},\dots,(L+1)^0,(L+2)^{f_{L+2}})$, where $k+\sum_{j=3}^{(L-1)/2}f_{2j}>0$ and $\alpha =0$ or 1,\\[-1.5ex]
\item $(2^{1},3^{3+f_3},\dots)$, where $\sum_{i=2}^{(L-1)/2} f_{2i} =0$.
\end{enumerate}

Using these cases one can rewrite $H_{L,2}(q)$ analytically with manifestly non-negative generating functions (recall \eqref{non_neg_products}).
\begin{align}\label{long_formula} q^3& + \frac{q^2(1-q^{L+1})}{(q^2;q)_{L+1}} - \left(\frac{1}{(q^3;q)_{L}} - 1\right)\\\nonumber&= \sum_{j=2}^{(L-1)/2} \frac{q^{2j} }{(q^3;q^2)_{(L+1)/2}(q^4;q^2)_{j-1}}\left( \frac{1}{(q^{2j+2};q^2)_{(L+1)/2-j}} -1 \right) + \frac{q^2}{(q^{L+1},q)_2} + \sum_{j=4}^L \frac{q^{3j+2}}{(q^j,q)_{L+3-j}}\\\nonumber&+ \sum_{j=(L+5)/2}^L \frac{q^{2j+2}}{(q^{j+1},q)_{L+2-j}} +\sum_{j=4}^{(L+3)/2}\frac{q^{2j+2}}{(q^{2j-1};q^2)_{(L+5)/2-j}} \left(\frac{1}{(q^{j+1};q)_{j-2}(q^{2j};q^2)_{(L+3)/2-j}}-1 \right)\\\nonumber
&+\sum_{j=2}^{(L-1)/2} \frac{q^{2j+3}}{(q^{2j+3};q^2)_{(L+1)/2-j}}\left(\frac{1}{(q^{2j+2};q^2)_{(L+1)/2-j}} -1\right)+\sum_{j=3}^{(L-1)/2} \frac{q^{2j+2}}{(q^{2j+1};q)_{L+2-2j}}\\\nonumber
&+ \frac{q^6}{(q^5;q^2)_{(L-1)/2}}\left(\frac{1}{(q^6;q^2)_{(L-3)/2}}-1 \right) + \frac{q^{13}}{(q^3;q)_L} + \frac{q^5+q^9}{(q^3;q^2)_{(L+1)/2}}\left( \frac{1}{(q^6;q^2)_{(L-3)/2}}-1 \right)\\\nonumber
&+\frac{q^{11}}{(q^3;q^2)_{(L+1)/2}}.
\end{align}
On the right-hand side of \eqref{long_formula}, the $k$-th term is the generating function for the number of partitions from the $k$-th $\Gamma_1$ unmapped case described above, where $k\in\{1,2,\dots,11\}$. Also note that as $L$ tends to infinity \eqref{long_formula} simplifies significantly, and can be reduced to \eqref{q3_F2_open_formula} after some labor.

The interested reader can also write $H_{L,2}(q)$ with only non-negative terms for even choices of $L$ with the same type of argument for $\Gamma_2^*$ and $\Gamma_2$ injections.

\section{An Alternative Proof of Theorem~\ref{Ordinary_Partitions_Combinatorial_Weighted_Theorem} and a Proof of Theorem~\ref{Smallest_part_positivity_THM}}\label{Section_Proof_of_First_theorem}

We start this section by recalling Theorem~\ref{Ordinary_Partitions_Combinatorial_Weighted_Theorem}:

\noindent \textbf{Theorem~\ref{Ordinary_Partitions_Combinatorial_Weighted_Theorem}.}
\begin{equation*}
\sum_{\pi\in\U} (-1)^{s(\pi)+1} q^{|\pi|} = \sum_{\pi\in\U} t(\pi) q^{|\pi|}.
\end{equation*}

Theorem~\ref{Ordinary_Partitions_Combinatorial_Weighted_Theorem} ---though proven by Jackson's transformation in \cite{BerkovichUncu4}--- can also be proven using the analytic generating functions for the partitions into distinct parts counted with $\pm 1$ weights depending on the parity of their ranks. Observe that \begin{equation}\label{rank_pf}\sum_{\pi\in\mathcal{D}} (-1)^{r(\pi)}q^{|\pi|}  = \sum_{n\geq 0} (-1)^n (q;q)_n q^{n+1}= \sum_{n\geq 1} \frac{q^{n(n+1)/2}}{(-q;q)_{n}},\end{equation} where $\mathcal{D}$ is the set of all partitions into distinct parts (where in a partition every frequency is either 0 or 1) with positive norm. Dividing both sides of the latter equality of \eqref{rank_pf} with the $q$-factorial $(q;q)_\infty$, the right-hand side series of \eqref{rank_pf} becomes the right-hand side of \eqref{Ordinary_Partitions_Analytic_Identity}. Also recall \begin{equation}\label{auxiliary_gen_func_line}\sum_{\pi\in\mathcal{U}} (-1)^{s(\pi)+1} q^{|\pi|} =\sum_{n\geq 1} (-1)^{n+1}\frac{q^n}{(q^n;q)_\infty} = \sum_{n\geq 1} \frac{q^n}{1+q^n}\frac{ 1}{(q;q)_{n-1}}.\end{equation}  The middle-term of \eqref{rank_pf} after the division with $(q;q)_\infty$ is the same as the middle term of \eqref{auxiliary_gen_func_line}. These observations together yield \eqref{Ordinary_Partitions_Analytic_Identity} and prove the Theorem~\ref{Ordinary_Partitions_Combinatorial_Weighted_Theorem}. The far right series in \eqref{rank_pf} first arose in Ramanujan's lost notebook and has been discussed in detail in \cite{Andrews_quote_lost_V} and \cite{AndrewsInvenciones}.

We would also like to remind the reader of the non-negativity question:\\

\noindent \textbf{Theorem~\ref{Smallest_part_positivity_THM}.} 
For $L\geq 1$,
\begin{equation*}
\sum_{\substack{\pi\in\mathcal{U}\\l-s\leq L}} (-1)^{s+1}q^{|\pi|} = \sum_{s\geq 1} \frac{(-1)^{s+1}q^s}{(q^s;q)_{L+1}}\succcurlyeq 0.
\end{equation*}
We define the following difference of generating functions \begin{equation}\label{FL_def}G_{L,1}(q) := \sum_{\substack{\pi\in\mathcal{U}\\s=1\\l-s\leq L}} q^{|\pi|} - \sum_{\substack{\pi\in\mathcal{U}\\s\geq 2\\l-s\leq L}} q^{|\pi|}. \end{equation} The closed analytic formulations of the two generating functions on the right-hand side of \eqref{FL_def} can be easily explained. All the partitions counted by the first generating function, \[ \sum_{\substack{\pi\in\mathcal{U}\\s=1\\l-s\leq L}} q^{|\pi|}, \] in \eqref{FL_def} has 1 as their smallest part and the largest part of these partitions can be at most $L+1$ due to the difference condition between the largest and the smallest parts. Therefore, we have \begin{equation}\label{FL_first_sum}
\sum_{\substack{\pi\in\mathcal{U}\\s=1\\l-s\leq L}} q^{|\pi|} = \frac{q}{(q;q)_{L+1}}.
\end{equation} For the second generating function of \eqref{FL_def}, we formulate the generating function as a sum over the number of parts. Let $\pi$ be a partition into $n$, parts where the smallest part $\geq 2$. We can clearly understand that $|\pi|\geq 2n$ since there are $n$ parts and all the parts are $\geq 2$. The whole column over the smallest part of the partition $\pi$ is generated by the $q$-factor \[\frac{q^{2n}}{1-q^n}.\] Stripping the column of the smallest part from the far left of the Ferrers diagram of $\pi$, we are left with a new partition with $\leq n-1$ parts, where the largest part is $\leq L$. These partitions are generated by the $q$-binomial coefficient \[{L+(n-1)\brack n-1}_q:= \frac{(q;q)_{L+(n-1)}}{(q;q)_{L} (q;q)_{n-1}}.\] Hence, putting these together, the analytic formula of the second sum in \eqref{FL_def} is \begin{equation}\label{FL_second_sum}
\sum_{\substack{\pi\in\mathcal{U}\\s\geq2\\l-s\leq L}} q^{|\pi|} = \sum_{n\geq 1} \frac{q^{2n}}{1-q^n} {L+(n-1)\brack n-1}_q.
\end{equation}

Putting \eqref{FL_first_sum} and \eqref{FL_second_sum} in \eqref{FL_def}, we get the following formula: \begin{equation}\label{F_analytic}
G_{L,1}(q) = \frac{q}{(q;q)_{L+1}} - \sum_{n\geq 1 }\frac{q^{2n}}{1-q^n} {L+(n-1) \brack (n-1)}_q.
\end{equation}

We can relate the $G_{L,1}(q)$ function with the $H_{L,1}(q)$ and also talk about its non-negativity.

\begin{theorem}\label{F_positivity_theorem}
For $L\geq 1$, \begin{equation}\label{F_H_relation}G_{L,1}(q) = \frac{H_{L,1}(q)}{1-q^L} \succcurlyeq 0\end{equation}
\end{theorem}

\begin{proof}
We start by showing the functional relation between $G_{L,1}(q)$ and $H_{L,1}(q)$. Comparing the right-hand sides of \eqref{H_original} and \eqref{F_analytic}, it is obvious that the first terms satisfy the claimed relation. Then, the problem reduces to justifying \[ \sum_{n\geq 1 }\frac{q^{2n}}{1-q^n} {L+(n-1) \brack (n-1)}_q=\frac{1}{1-q^L}\left(\frac{1}{(q^2;q)_L} -1 \right) .\] Observe that \begin{equation}\label{important_shift}\frac{1}{1-q^n} {L+(n-1) \brack (n-1)}_q = \frac{1}{1-q^L} {(L-1)+n \brack n}_q.\end{equation} Applying \eqref{important_shift} and the $q$-binomial theorem \eqref{q_binomial}, we can verify the formula of $G_{L,1}(q)$: \begin{align*}
\sum_{n\geq 1 }\frac{q^{2n}}{1-q^n} {L+(n-1) \brack (n-1)}_q &= \frac{1}{1-q^L}\sum_{n\geq 1 }q^{2n} {(L-1)+n \brack n}_q\\
&=  \frac{1}{1-q^L}\left(-1 + \sum_{n\geq 0 }q^{2n} \frac{(q^L;q)_n}{(q;q)_n}\right) \\
&=  \frac{1}{1-q^L}\left(-1 + \frac{1}{(q^2;q)_L}\right).
\end{align*}
The positivity claim on $G_{L,1}(q)$, for $L\geq 2$ follows from Theorem~\ref{H_positivity_theorem} as $1/(1-q^L)$ and $H_{L,1}(q)$ both have non-negative series, their multiplication has non-negative series. The $L=1$ case can be directly/algebraically checked from \eqref{F_analytic}. For $L=1$, the expression \eqref{F_analytic} reduces to $q/(1-q^2)$ and hence, is represented by a power series with non-negative coefficients.
\end{proof}

Theorem~\ref{F_positivity_theorem} implies Theorem~\ref{Smallest_part_positivity_THM}. 

\begin{proof} \textit{(Theorem~\ref{Smallest_part_positivity_THM}) }Observe that \begin{equation}\label{abstract_positivity_of_alternating_smallest_part}\sum_{\substack{\pi\in\mathcal{U}\\l-s\leq L}} (-1)^{s+1}q^{|\pi|} = G_{L,1}(q) + 2\cdot\sum_{\substack{\pi\in\mathcal{U}\\s>1\\ s\equiv 1 (\text{mod}\,2)\\l-s\leq L}} q^{|\pi|},\end{equation} by the original definition of $G_{L,1}(q)$, equation \eqref{FL_def}. For $L\geq 1$, we have $G_{L,1}(q)\succcurlyeq 0$ from Theorem~\ref{F_positivity_theorem}. It is also clear that the second sum of \eqref{abstract_positivity_of_alternating_smallest_part}, \[\sum_{\substack{\pi\in\mathcal{U}\\s>1\\ s\equiv 1 (\text{mod}\,2)\\l-s\leq L}} q^{|\pi|},\] is also non-negative. Hence, we get our claim \begin{equation}\tag{\ref{weighted_sum_of_bounded_difference}}\sum_{\substack{\pi\in\mathcal{U}\\l-s\leq L}} (-1)^{s+1}q^{|\pi|}\succcurlyeq 0.\end{equation} \end{proof}

One can also define the analogous function \[G_{L,2}(q) := \sum_{\substack{\pi\in\mathcal{U}\\s=2\\l-s\leq L}} q^{|\pi|} - \sum_{\substack{\pi\in\mathcal{U}\\s\geq 3\\l-s\leq L}} q^{|\pi|},\] which can be written analytically as \[G_{L,2}(q) = \frac{q^2}{(q^2;q)_{L+1}} - \sum_{n\geq 1 }\frac{q^{3n}}{1-q^n} {L+(n-1) \brack (n-1)}_q.\] Keeping in mind the identity \eqref{important_shift} as it is used in the proof of Theorem~\ref{F_positivity_theorem}, it is easy to prove the following theorem:

\begin{theorem} For $L\geq 3$, \[G_{L,2}(q) = \frac{H^*_{L,2}(q)}{1-q^L},\] where \[H^*_{L,2}(q):= \frac{q^2(1-q^L)}{(q^2;q)_{L+1}} - \left(\frac{1}{(q^3;q)_L} -1 \right).\]
\end{theorem}

With this definition, we can make the similar claim to Conjecture~\ref{Sets_Conjecture2} about $G_{L,2}(q)$. 
\begin{conjecture}\label{F_2_conj} For $L= 3$ and 4, \[G_{L,2}(q) + q^3 + q^9\succcurlyeq 0 ,\] and for $L\geq 5$ \[G_{L,2}(q)+q^3\succcurlyeq 0.\]
\end{conjecture}

A more general analytic conjecture, which contains Conjecture~\ref{F_2_conj}, is discussed in Section~\ref{Section_Outlook}.

\section{Transformations of the Analytic Refined Weighted Identity}\label{Section_Refinement}

We now shift our focus to the analytic version of Theorem~\ref{Smallest_part_positivity_THM}. The sum in the statement \eqref{weighted_sum_of_bounded_difference} can be written in an equivalent analytical form \begin{equation}
\label{analytic_weighted_sum_of_bounded_difference} \sum_{\substack{\pi\in\mathcal{U}\\l-s\leq L}} (-1)^{s+1}q^{|\pi|} = \sum_{n\geq 1} \frac{q^n}{1+q^n} {L + (n-1) \brack (n-1) }_q,
\end{equation} similar to the discussion of \eqref{FL_second_sum}. The connection between expressions \eqref{weighted_sum_of_bounded_difference} and \eqref{analytic_weighted_sum_of_bounded_difference} is the first Heine transformation \eqref{Heine1}, which will be more openly discussed after the following theorem.

One can also write the last term of the right-hand side of \eqref{abstract_positivity_of_alternating_smallest_part} by focusing on the smallest parts \begin{equation}\label{odds_more_than_one}\sum_{\substack{\pi\in\mathcal{U}\\s>1\\ s\equiv 1 (\text{mod}\,2)\\l-s\leq L}} q^{|\pi|} = \sum_{n\geq 1} \frac{q^{2n+1}}{(q^{2n+1};q)_{L+1}}. \end{equation} Rewriting the terms analytically in \eqref{abstract_positivity_of_alternating_smallest_part} by plugging the expression \eqref{H_secondary} in \eqref{F_H_relation}, and copying \eqref{analytic_weighted_sum_of_bounded_difference} and \eqref{odds_more_than_one} in their respective places yields:

\begin{theorem}
\begin{align}
\sum_{n\geq 1} \frac{q^n}{1+q^n} {L + (n-1) \brack (n-1) }_q
\\\nonumber = \sum_{s=2}^{L} \frac{q^{2s+1}}{(1-q^L)(q^s;q)_{L+2-s}}  &+ \frac{q}{(1-q^L)(1-q^{L+1})} + \frac{q^3(1-q^{L-2})}{(1-q^L)(q;q)_{L+1}} + 2\cdot \sum_{n\geq 1} \frac{q^{2n+1}}{(q^{2n+1};q)_{L+1}}.
\end{align}
\end{theorem}

Moreover, one can write the expression \[ \sum_{n\geq 1} \frac{q^n}{1+q^n} {L + (n-1) \brack (n-1) }_q\] $q$-hypergeometrically. Note that \[\frac{1+q}{1+q^n} = \frac{(-q;q)_{n-1}}{(-q^2;q)_{n-1}}.\] Now it is easy to see that  \begin{equation}\label{qHypergeometric_l_s_less_than_L}\sum_{n\geq 1} \frac{q^n}{1+q^n} {L + (n-1) \brack (n-1) }_q = \frac{q}{1+q} {}_2 \phi_1 \left( \genfrac{}{}{0pt}{}{q^{L+1},\ -q}{-q^2} \; ; q,\ q \right).\end{equation} Applying the first Heine transformation \eqref{Heine1} to this ${}_2\phi_1$, we get the $q$-series of \eqref{weighted_sum_of_bounded_difference}.

Applying the Jackson transformation \eqref{Jackson_transformation} to \eqref{qHypergeometric_l_s_less_than_L} gives the identity \begin{equation}\label{q_hypergeometric_Jackson_transformation}\frac{q}{1+q} {}_2 \phi_1 \left( \genfrac{}{}{0pt}{}{q^{L+1},\ -q}{-q^2} \; ; q,\ q \right)= \frac{q}{1+q}\frac{(q^{L+2};q)_\infty}{(q;q)_\infty} {}_2 \phi_2 \left( \genfrac{}{}{0pt}{}{q^{L+1},\ q}{-q^2,\ q^{L+2}} \; ; q,\ -q^2 \right).\end{equation} After elementry simplifications and shifting the summation variable $n\mapsto n-1$, we arrive at the identity \begin{equation}\label{Jackson_transformed_equation}\sum_{n\geq 1} \frac{q^n}{1+q^n} {L + (n-1) \brack (n-1) }_q = \frac{1}{(q;q)_L} \sum_{n\geq 1} \frac{q^{n+1\choose 2}}{(-q;q)_n(1-q^{L+n})}.\end{equation} The right-hand side expression in \eqref{Jackson_transformed_equation}, similar to \eqref{analytic_weighted_sum_of_bounded_difference}, is not manifestly positive at first sight, but Theorem~\ref{Smallest_part_positivity_THM} carries over and proves the positivity.

\begin{theorem}\label{Jackson_transformed_positivity} For $L\geq 1$, 
\begin{equation*} \frac{1}{(q;q)_L} \sum_{n\geq 1} \frac{q^{n+1\choose 2}}{(-q;q)_n(1-q^{L+n})} \succcurlyeq 0.\end{equation*}
\end{theorem}

One more interesting equivalent expression to the ones of \eqref{Jackson_transformed_equation} ---still not manifestly positive--- is the outcome of the third Heine transformation \eqref{Heine3} of the left side of \eqref{q_hypergeometric_Jackson_transformation}. After the necessary simplifications and a shift in summation, the Heine transformation \eqref{Heine3} gives \begin{equation}\label{intermediate_term}\frac{1}{(q;q)_L}\sum_{n\geq1}\frac{(-q^{1-L};q)_{n-1}}{(-q;q)_n}q^{L(n-1)+n},\end{equation} as an equal summation to that of \eqref{Jackson_transformed_equation}. The expression \eqref{intermediate_term} is interesting in its own right. It is clear that it has different summands than either side of \eqref{Jackson_transformed_equation}. Yet, when $L=0$, it matches term-by-term with the left-hand side expression \[\sum_{n\geq 1} \frac{q^n}{1+q^n} {L + (n-1) \brack (n-1) }_q,\] evaluated at $L=0$, and when $L\rightarrow\infty$, it matches term-by-term the right-hand side \[\frac{1}{(q;q)_L} \sum_{n\geq 1} \frac{q^{n+1\choose 2}}{(-q;q)_n(1-q^{L+n})},\] as $L\rightarrow\infty$ in the identity \eqref{Jackson_transformed_equation}. Therefore, to this extent, \eqref{intermediate_term} is the intermediate term in \eqref{Jackson_transformed_equation}.

\section{Outlook}\label{Section_Outlook}
One project to pursue is to identify the statistics $t_L(\pi)$ for partitions, which would be the refined statistics of $t(\pi)$ of Theorem~\ref{Ordinary_Partitions_Combinatorial_Weighted_Theorem} for partitions with the difference between the largest and the smallest parts bounded by $L$. As it stands, going from Theorem~\ref{Ordinary_Partitions_Combinatorial_Weighted_Theorem} to Theorem~\ref{Smallest_part_positivity_THM} we lose grasp of the non-negative statistics $t(\pi)$.

Another question is related to the $H_{L,1}(q)$ and $H_{L,2}(q)$ functions of Section~\ref{Section_non_negativity}. Recall that a series $\sum_{n\geq 0} a_n q^n$ is called \textit{eventually positive} if there is some $k$ such that $a_n >0$ for all $n > k$. Theorem~\ref{H_positivity_theorem} and Theorem~\ref{H_positivity_theorem2} seem to be the initial steps of a eventually positive family of $q$-products. Let \begin{equation*} H_{L,s,k} (q) = \frac{q^s(1-q^{k})}{(q^s;q)_{L+1}} - \left(\frac{1}{(q^{s+1};q)_{L}} - 1\right),\end{equation*} then we have the following claim:

\begin{conjecture}\label{H_conj} For $L$ and $k\geq s+1$,
\begin{center} $H_{L,s,k}(q)$ is eventually positive. \end{center}
\end{conjecture}

We have already proven the conjecture for the $(L,s,k)=(L,1,L)$ and $(L,2,L+1)$ families in Theorem~\ref{H_positivity_theorem} and \ref{H_positivity_theorem2}. The particular branch $H_{L,s,L+s-1}(q)$ is a natural generalization of the functions $H_{L,1}$ and $H_{L,2}$ mentioned in Section~\ref{Section_non_negativity}, and the non-negativity claim related to Conjecture~\ref{Sets_Conjecture}. All other triplets with $L=k\geq s+1$ are related to the Conjecture~\ref{Sets_Conjecture2} and Conjecture~\ref{F_2_conj}. Therefore, one can view Conjecture~\ref{H_conj} with the above relations as natural extension of these observations. For all other triplets $(L,s,k)$ are experimental.

The number of exceptional cases increases with $s$, making it less feasible to combinatorially study these functions for larger starting values $s$. More interestingly, the presence of a one-time exception at $q^9$ for the $(L,s,k)=(4,2,5)$ case ($H_{L,2}(q)$), which was handled in Theorem~\ref{Ali_comb_THM2}, also hints a higher degree of underlying complexity.

\section{Acknowledgment}

Authors would like to thank George Andrews for interest and helpful insights.  We are grateful to William Severa for his careful reading of the manuscript.

%

\end{document}